\documentclass[11pt]{article}
\usepackage{amsmath}
\usepackage{amssymb}
\usepackage{amsthm}
\usepackage[usenames]{color}
\usepackage{amscd}

\usepackage[colorlinks=true,linkcolor=blue,filecolor=red,
citecolor=webgreen]{hyperref}
\definecolor{webgreen}{rgb}{0,.5,0}

\def\N{{\Bbb N}}
\def\Z{{\Bbb Z}}

\def\C{{\Bbb C}}

\def\1{{\bf 1}}

\newtheorem{theorem}{Theorem}[section]
\newtheorem{lemma}{Lemma}[section]
\newtheorem{corollary}{Corollary}[section]

\newtheorem{remark}{Remark}[section]

\begin{document}

\title{{\bf Trigonometric Representations of Generalized \\ Dedekind and Hardy Sums via the Discrete \\ Fourier Transform \thanks{In:  Analytic Number Theory - In Honor of Helmut Maier's 60th Birthday, C.~Pomerance and M.~Th.~Rassias (eds.), Springer, New York,
2015, pp. 329--343.}}}
\author{Michael Th. Rassias and L\'aszl\'o T\'oth}
\date{}
\maketitle

\centerline{\it To Professor Helmut Maier on his 60th birthday}

\vskip3mm

\abstract{We introduce some new higher dimensional generalizations of the Dedekind sums associated with the Bernoulli functions and
of those Hardy sums which are defined by the sawtooth function. We generalize a variant of Parseval's formula for the discrete Fourier
transform to derive finite trigonometric representations for these sums in a simple unified manner. We also consider a related sum involving the
Hurwitz zeta function.}

\medskip



\section{Introduction}

The classical Dedekind sum is defined for $h\in \Z$, $k\in \N:=\{1,2,\ldots\}$ by
\begin{equation*}
s(h,k):=\sum_{a \text{ (mod $k$)}} \left(\left(\frac{a}{k}\right)\right) \left(\left(\frac{ah}{k}\right)\right),
\end{equation*}
adopting the usual notation
\begin{equation*}
\left(\left(x\right)\right) :=\begin{cases} \{x\} - \frac{1}{2}, & \text{ if } \ x\not\in \Z, \\
0, & \text{ if } \ x\in \Z, \end{cases}
\end{equation*}
where $\{x\}:=x- \lfloor x \rfloor$ stands for the fractional part of $x$ {{{{(cf. \cite{IKW}, \cite{RadGro1972})}}}}.
If $\gcd(h,k)=1$, then $s(h,k)$ can be represented as
\begin{equation} \label{cot_repres}
s(h,k)= \frac1{4k} \sum_{a=1}^{k-1} \cot \left(\frac{\pi a}{k}\right) \cot \left(\frac{\pi ah}{k}\right)
\end{equation}
and
\begin{equation} \label{inf_repres}
s(h,k) = \frac1{2\pi} \sum_{\substack{r=1\\ r\not\equiv 0 {
\text{ (mod $k$)}}}}^{\infty} \frac1{r} \cot \left(\frac{\pi rh}{k} \right).
\end{equation}

The identities \eqref{cot_repres} and \eqref{inf_repres} were derived in 1933 by H.~Rademacher \cite{Rad1933}
in order to obtain a simple direct proof of the reciprocity formula for the Dedekind sums. See
also \cite[pp.\ 18--25]{RadGro1972}. According to \cite[p.\ 347]{Alm1998}, \eqref{cot_repres}
was obtained earlier, in 1923 by H.~Mellin. The identity \eqref{cot_repres} is also the starting point for
various generalizations of $s(h,k)$. See, e.g., the papers of
M.~Beck \cite{Bec2003}, U.~Dieter \cite{Die1984}, D.~Zagier
\cite{Zag1973}.

It is known that \eqref{cot_repres} is a direct consequence of a variant of Parseval's formula for the discrete
Fourier transform (DFT). See the paper of G.~Almkvist \cite[Sect.\ 6]{Alm1998} and the book by M.~Beck and S.~Robins
\cite[Ch.\ 7]{BecRob2007}. More specifically, consider a function $f:\Z \to \C$, which is $k$-periodic
(periodic with period $k$), where $k\in \N$. We define the DFT of
$f$ as the function $\widehat{f}={\cal F}(f)$, given by
\begin{equation*}
\widehat{f}(n):= \sum_{a \text{ (mod $k$)}} f(a) e^{-2\pi i an/k}
\quad (n\in \Z).
\end{equation*}

Furthermore, if $f_1$ and $f_2$ are $k$-periodic functions, then
their inner product is
\begin{equation*}
\langle f_1,f_2\rangle := \sum_{a \text{ (mod $k$)}} f_1(a)\overline{f_2(a)},
\end{equation*}
having the property
\begin{equation*}
\langle f_1,f_2\rangle = \frac1{k} \langle \widehat{f_1}, \widehat{f_2} \rangle,
\end{equation*}
or equivalently,
\begin{equation} \label{Parseval}
\sum_{a \text{ (mod $k$)}} f_1(a)f_2(-a) = \frac1{k} \sum_{a \text{ (mod $k$)}} \widehat{f_1}(a) \widehat{f_2}(a).
\end{equation}

Now, \eqref{cot_repres} follows by applying \eqref{Parseval} to the
functions $$f_1(a)=\left(\left(\frac{a}{k}\right)\right),\ \ f_2(a)=
\left(\left(\frac{ah}{k}\right)\right)$$ and using {{{{the fact}}}} that the DFT of
the sawtooth function is essentially the cotangent function.

It is the aim of this paper to exploit this idea in order to deduce similar finite trigonometric representations for certain
new generalized Dedekind and Hardy sums, in a simple unified manner. Our results are direct applications of a higher dimensional version of
the identity \eqref{Parseval}, included in Theorem \ref{th_main}. We derive in this way
Zagier-type identities for new higher dimensional generalizations of the Dedekind sums associated
to the Bernoulli functions and of those Hardy sums which are defined by the sawtooth function.
Note that all finite trigonometric representations we obtain contain only the tangent and cotangent functions, and are special cases of the
Dedekind cotangent sums investigated by M.~Beck \cite{Bec2003}. Therefore the reciprocity law proved in \cite[Th.\ 2]{Bec2003} can be applied
for each sum.

Furthermore, we consider a related sum, studied by M.~Mikol\'as \cite{Mik1957}, involving the Hurwitz
zeta function.  We remark that \eqref{Parseval} was used to
evaluate some finite trigonometric and character sums, but not Dedekind and related sums, by M.~Beck and
M.~Halloran \cite{BecHal2010}. We point out that the identity \eqref{inf_repres} can be obtained from
\eqref{cot_repres} using another general result (Lemma \ref{lemma_2} in Section \ref{Sect_Final_remarks}).


\section{Properties of the DFT}

We will apply the following general result.

\begin{theorem} \label{th_main} Let $f_1,\ldots,f_m$ be arbitrary
$k$-periodic functions and let $h_j \in \Z$, $\gcd(h_j,k) =1$ {\rm ($1\le j\le m$)}, where $m,k\in \N$. Then
\begin{equation*}
\sum_{\substack{a_1,\ldots,a_m \text{ {\rm (mod $k$)}} \\ a_1+\ldots+a_m\equiv 0 \text{ {\rm (mod $k$)}}}} f_1(a_1h_1)\cdots f_m(a_m h_m) = \frac1{k}
\sum_{a \text{ {\rm (mod $k$)}}} \widehat{f_1}(ah_1')\cdots \widehat{f_m}(ah_m'),
\end{equation*}
where $h_j'$ is the multiplicative inverse of $h_j$ {\rm (mod $k$)}, that is $h_jh_j'\equiv 1$ {\rm (mod $k$) with $1\le j\le m$}.
\end{theorem}

\begin{proof}
We only need some simple well known facts concerning the DFT. See, for instance
\cite[Ch.\ 2]{Ter1999}, \cite[Ch.\ 7]{BecRob2007}. The Cauchy convolution of the
$k$-periodic functions $f_1$ and $f_2$ is defined by
\begin{equation*}
(f_1\otimes f_2)(n):= \sum_{\substack{a_1,a_2 \text{ {\rm (mod $k$)}} \\ a_1+a_2\equiv n \text{ {\rm (mod $k$)}}}} f_1(a_1)f_2(a_2) =
\sum_{a \text{ {\rm (mod $k$)}}} f_1(a)f_2(n-a)  \quad (n\in \Z),
\end{equation*}
which is associative and commutative. Also, $$\widehat{f_1\otimes
f_2} = \widehat{f_1}\widehat{f_2}.$$
More generally, if
$f_1,\ldots,f_m$ are $k$-periodic functions, then $${\cal
F}(f_1\otimes \cdots \otimes f_m) = {\cal F}(f_1) \cdots {\cal
F}(f_m).$$ Recalling that $${\cal F}({\cal F}(f))(n)= k f(-n) \quad (n\in
\Z),$$
{{{{which is}}}} valid for every $k$-periodic $f$, this {{{{yields}}}}
\begin{equation*}
k(f_1\otimes \cdots \otimes f_m)(-n) ={\cal F}({\cal F}(f_1) \cdots
{\cal F}(f_m))(n),
\end{equation*}
that is
\begin{equation*}
\sum_{\substack{a_1,\ldots,a_m \text{ {\rm (mod $k$)}} \\ a_1+\ldots+a_m\equiv -n \text{ {\rm (mod $k$)}}}} f_1(a_1)\cdots f_m(a_m) = \frac1{k}
\sum_{a \text{ {\rm (mod $k$)}}}  \widehat{f_1}(a)\cdots \widehat{f_m}(a) e^{-2\pi i an/k}
\quad (n\in \Z).
\end{equation*}

For $n=0$ we obtain
\begin{equation} \label{simple}
\sum_{\substack{a_1,\ldots,a_m \text{ {\rm (mod $k$)}} \\ a_1+\ldots+a_m\equiv 0 \text{ {\rm (mod $k$)}}}} f_1(a_1)\cdots f_m(a_m) = \frac1{k}
\sum_{a \text{ {\rm (mod $k$)}}}  \widehat{f_1}(a)\cdots \widehat{f_m}(a).
\end{equation}

Now the result follows from \eqref{simple} by showing the following
property: Let $f$ be a $k$-periodic function and let $h\in \Z$ such
that $\gcd(h,k)=1$. Then the DFT of the function $g$ defined by
$g(n)=f(nh)$ ($n\in \Z$) is $\widehat{g}(n)=\widehat{f}(nh')$ ($n\in
\Z$), where $h'$ is the multiplicative inverse of $h$ (mod $k$).

Indeed,
\begin{equation*}
\widehat{g}(n)= \sum_{a \text{ {\rm (mod $k$)}}} g(a) e^{-2\pi i
an/k} = \sum_{a \text{ {\rm (mod $k$)}}} f(ah) e^{-2\pi i ahh'n/k},
\end{equation*}
and since $\gcd(h,k)=1$, if $a$ runs through a complete system of residues (mod $k$), then so does $b=ah$. Therefore,
\begin{equation*}
\widehat{g}(n)= \sum_{b \text{ {\rm (mod $k$)}}} f(b) e^{-2\pi i
bh'n/k} = \widehat{f}(nh').
\end{equation*}
\end{proof}

\begin{corollary} \label{cor_m_2_gen} Let $f_1$ and $f_2$ be $k$-periodic functions {\rm ($k \in \N$)} and
let $h_1,h_2\in \Z$, $\gcd(h_1,k)=\gcd(h_2,k)=1$. Then
\begin{equation} \label{Dedekind_m_2_general}
\sum_{a \text{ {\rm (mod $k$)}}} f_1(ah_1)f_2(ah_2) = \frac1{k}
\sum_{a \text{ {\rm (mod $k$)}}} \widehat{f_1}(-ah_2)
\widehat{f_2}(ah_1).
\end{equation}
\end{corollary}

\begin{proof} Apply Theorem \ref{th_main} for $m=2$.  We deduce that
\begin{align*}
\sum_{a \text{ {\rm (mod $k$)}}} f_1(ah_1)f_2(-ah_2) & = \frac1{k}
\sum_{a \text{ {\rm (mod $k$)}}} \widehat{f_1}(ah_1')
\widehat{f_2}(ah_2') \\
& = \frac1{k} \sum_{a \text{ {\rm (mod $k$)}}} \widehat{f_1}(ah_1'h_2'h_2)
\widehat{f_2}(ah_1'h_2'h_1) \\
& = \frac1{k} \sum_{b \text{ {\rm (mod $k$)}}} \widehat{f_1}(bh_2)
\widehat{f_2}(bh_1),
\end{align*}
by using the fact that if $a$ runs through a complete system of residues (mod $k$), then so does $b=ah_1'h_2'$, since $\gcd(h_1h_2,k)=1$.
This gives \eqref{Dedekind_m_2_general} by setting $h_2:=-h_2$.
\end{proof}

For $h_1=1$, $h_2=-1$ from \eqref{Dedekind_m_2_general} we {{{{derive}}}}
\eqref{Parseval}.

\begin{corollary} \label{coroll_2} Let $f_1$ and $f_2$ be $k$-periodic functions {\rm ($k\in \N$)} and assume that
$f_1$ or $f_2$ is odd (resp. even).
Let $h_1,h_2\in \Z$, $\gcd(h_1,k)=\gcd(h_2,k)=1$. Then
\begin{equation} \label{Dedekind_m_2}
\sum_{a \text{ {\rm (mod $k$)}}} f_1(ah_1)f_2(ah_2) = \frac{(-1)^s}{k}
\sum_{a \text{ {\rm (mod $k$)}}} \widehat{f_1}(ah_2) \widehat{f_2}(ah_1),
\end{equation}
where $s=1$ if $f_1$ or $f_2$ is odd, $s=0$ if $f_1$ or $f_2$ is even.
\end{corollary}

Note that if the function $f$ is odd (resp. even), then $\widehat{f}$ is
also odd (resp. even). If one of the functions $f_1, f_2$ is odd and the
other one is even, then both sides of \eqref{Dedekind_m_2} are zero.

In this paper we will use the following DFT pairs of $k$-periodic functions.

\begin{lemma} \label{lemma} {\rm (i)} Let $k\in \N$. The DFT of the $k$-periodic odd function
$f(n)=\left(\left( \frac{n}{k} \right) \right)$ {\rm ($n\in \Z$)} is
\begin{equation*}
\widehat{f}(n)= \begin{cases} \frac{i}{2}\cot \left( \frac{\pi n}{k}\right), & \text{ if } \ k\nmid n, \\ 0, & \text{ if } \ k\mid n.
\end{cases}
\end{equation*}

{\rm (ii)} Let $k\in \N$ and let $\overline{B_r}$ {\rm ($r\in \N$)}
be the Bernoulli functions (cf. Section \ref{subsect_2_1}). The DFT
of the $k$-periodic function
$f(n)=\overline{B}_r\left(\frac{n}{k}\right)$ {\rm ($n\in\Z$)} is
\begin{equation*}
\widehat{f}(n) = \begin{cases} rk^{1-r} \left(\frac{i}{2}\right)^r
\cot^{(r-1)}\left(\frac{\pi n}{k}\right), & \text{ if } \ k\nmid n, \\ B_r
k^{1-r}, & \text{ if } \ k\mid n,
\end{cases}
\end{equation*}
where $B_r$ is the $r$-th Bernoulli number and $\cot^{(m)}$ is the
$m$-th derivative of the cotangent function.

{\rm (iii)} Let $k\in \N$ be even. The DFT of the $k$-periodic odd function $f(n)=(-1)^n \left(\left(\frac{n}{k} \right)\right)$
{\rm ($n\in\Z$)} is
\begin{equation*}
\widehat{f}(n)= \begin{cases} -\frac{i}{2} \tan \left(\frac{\pi n}{k}\right), & \text{ if } \ n\not\equiv \frac{k}{2} \text{ {\rm (mod $k$)}},
\\ 0, & \text{ if } \ n\equiv \frac{k}{2} \text{ {\rm (mod $k$)}}.
\end{cases}
\end{equation*}

{\rm (iv)} Let $k$ be odd and let $n \text{ {\rm (mod $k$)}}
=n-k\lfloor n/k\rfloor$ be the least nonnegative residue of $n$ {\rm
(mod $k$)}. The DFT of the $k$-periodic odd function
\begin{equation*}
f(n)=  \begin{cases}  (-1)^{n \text{ {\rm (mod $k$)}}}, & \text{ if
} \  k\nmid n,
\\ 0, & \text{ if } \ k \mid n
\end{cases}
\end{equation*}
is
\begin{equation*}
\widehat{f}(n)= i\tan \left(\frac{\pi n}{k}\right) \quad (n\in \Z).
\end{equation*}

{\rm (v)} Let $k\in \N$. Let $F(s,x)$, $\zeta(s,x)$ and $\zeta(s)$ be the periodic zeta function, the Hurwitz zeta function, and the
Riemann zeta function, respectively (cf. Section \ref{subsect_3_3}). For $\Re s>1$ the DFT of the $k$ periodic function
$f(n)= F(s,\frac{n}{k})$ {\rm ($n\in\Z$)} is
\begin{equation*}
\widehat{f}(n)=\begin{cases} k^{1-s} \zeta\left(s,\left\{\frac{n}{k}\right\}\right), & \text{ if } \ k\nmid n, \\
k^{1-s}\zeta(s), & \text{ if } \ k\mid n.
\end{cases}
\end{equation*}
\end{lemma}

Here (i) and (iv) are well known. They follow, together with (iii) and (v), by easy computations from the definition of the DFT.
For (ii) we refer to \cite[Lemma 6]{Bec2003}. See also \cite[Sect.\ 6]{Alm1998}.


\section{Applications}
\setcounter{theorem}{1}
\setcounter{corollary}{2}

\subsection{Generalized Dedekind sums} \label{subsect_2_1}

\noindent We first derive the following higher dimensional generalization of the identity
\eqref{cot_repres}, first deduced by D.~Zagier \cite[Th.\ p. 157]{Zag1973}, in a slightly different form, by applying some
other arguments.

\begin{theorem} Let $k\in \N$, $m\in \N$ be even and let $h_j\in \Z$, $\gcd(h_j,k)=1$ {\rm ($1\le j\le m$)}. Then
\begin{align} \label{higher_dim_Dedekind}
& \sum_{\substack{a_1,\ldots,a_m \text{ {\rm (mod $k$)}} \\ a_1+\ldots+a_m\equiv 0 \text{ {\rm (mod $k$)}}}}
\left(\left(\frac{a_1h_1}{k}\right)\right)
\cdots \left(\left(\frac{a_m h_m}{k}\right)\right)  \nonumber\\
& \qquad \qquad  =
\frac{(-1)^{m/2}}{2^m k} \sum_{a=1}^{k-1} \cot \left(\frac{\pi
ah_1'}{k}\right) \cdots \cot \left(\frac{\pi ah_m'}{k}\right).
\end{align}
\end{theorem}

\begin{proof} Apply Theorem \ref{th_main} for $f_1=\ldots=f_m=f$, where
$f(n)=\left(\left( \frac{n}{k} \right) \right)$ ($n\in \Z$). Use
Lemma \ref{lemma}/(i).
\end{proof}

Note that if $m$ is odd, then both sides of
\eqref{higher_dim_Dedekind} are zero.

\begin{corollary} {{{{Assume that}}}} $m=2$. Let $k\in \N$, $h_1,h_2\in \Z$, $\gcd(h_1,k)=\gcd(h_2,k)=1$. Then
\begin{equation} \label{cot_repres_homog}
\sum_{a=1}^{k-1}
\left(\left(\frac{ah_1}{k}\right)\right) \left(\left(\frac{ah_2}{k}\right)\right) =
\frac1{4k} \sum_{a=1}^{k-1} \cot \left(\frac{\pi
ah_1}{k}\right) \cot \left(\frac{\pi ah_2}{k}\right).
\end{equation}
\end{corollary}

Identity \eqref{cot_repres_homog} is the homogeneous version of
\eqref{cot_repres} and is equivalent to \eqref{cot_repres}.

Now consider the Bernoulli polynomials $B_r(x)$ ($r\ge 0$), defined
by
\begin{equation*}
\frac{te^{\, xt}}{e^{\, t}-1}= \sum_{r=0}^{\infty} \frac{B_r(x)}{r!}t^{\,r}.
\end{equation*}

Here $$B_1(x)=x-1/2,\ B_2(x)=x^2-x+1/6,\ B_3(x)=x^3-3x^2/2+x/2\ \text{and}\ B_r:=B_r(0)$$ are the Bernoulli numbers.
The Bernoulli functions $x\mapsto \overline{B}_r(x)$ are given by
\begin{equation*}
\overline{B}_r(x)=B_r(\{x\}) \quad (x\in {\Bbb R}).
\end{equation*}

\noindent Note that $$\overline{B}_1(x)=((x))\ \text{for}\ x\notin \Z,$$ but $$\overline{B}_1(x)=-1/2 \ne 0=((x))\ \text{for}\ x\in \Z.$$

For $r_1,\ldots,r_m \in \N$, $h_1,\ldots,h_m \in \Z$  we define the
higher dimensional Dedekind--Bernoulli sum by
\begin{equation} \label{Bern_gen_def}
s_{r_1,\ldots,r_m}(h_1,\ldots,h_m;k):=
\sum_{\substack{a_1,\ldots,a_m \text{ {\rm (mod $k$)}} \\
a_1+\ldots+a_m\equiv 0 \text{ {\rm (mod $k$)}}}}
\overline{B}_{r_1}\left(\frac{a_1h_1}{k}\right) \cdots
\overline{B}_{r_m}\left(\frac{a_mh_m}{k}\right).
\end{equation}
In the case {{{{when}}}} $m=2$ and by {{{{setting}}}} $h_2:=-h_2$ we {{{{obtain}}}} the sum
\begin{equation} \label{Ded_Bern_2}
s_{r_1,r_2}(h_1,-h_2;k):=
\sum_{a \text{ {\rm (mod $k$)}}} \overline{B}_{r_1}\left(\frac{ah_1}{k}\right)
\overline{B}_{r_2}\left(\frac{ah_2}{k}\right),
\end{equation}
first investigated by L. Carlitz \cite{Car1953} and M. Mikol\'as
\cite{Mik1957}. See the paper of M.~Beck \cite{Bec2003} for further
historical remarks.

\begin{theorem} \label{Th_Bern_gen} Let $k,m, r_j \in \N$ be such that
$A:=r_1+\ldots +r_m$ is even and $h_j\in \Z$, $\gcd(h_j,k)=1$
{\rm ($1\le j\le m$)}. Then
\begin{align} \nonumber
& \qquad \qquad \qquad s_{r_1,\ldots,r_m}(h_1,\ldots,h_m;k) =
\frac{B_{r_1}\cdots
B_{r_m}}{k^{A-m+1}} \\
\label{Bern_gen}
& + \frac{(-1)^{A/2} r_1\cdots r_m}{2^Ak^{A-m+1}}
\sum_{a=1}^{k-1} \cot^{(r_1-1)}\left(\frac{\pi
ah_1'}{k}\right)\cdots \cot^{(r_m-1)}\left(\frac{\pi
ah_m'}{k}\right),
\end{align}
\end{theorem}

Note that if $A$ is odd, then the sum in \eqref{Bern_gen} vanishes.
If $A$ is odd and there is at least one $j$ such that $r_j\ge 3$,
then $B_{r_j}=0$ and the sum \eqref{Bern_gen_def} vanishes as well.

\begin{proof} Apply Theorem \ref{th_main} and Lemma \ref{lemma}/(ii) to the functions
$$ f_j(n)=\overline{B}_{r_j}\left(\frac{n}{k}\right) \quad (1\le j\le m).$$
\end{proof}

\begin{corollary} {\rm (\cite[Cor.\ 7]{Bec2003})} Let $k, r_1,r_2 \in \N$, $h_1,h_2 \in \Z$ be such that
$r_1+r_2$ is even and $\gcd(h_1,k)=\gcd(h_2,k)=1$. Then
\begin{equation*}
\sum_{a \text{ {\rm (mod $k$)}}}
\overline{B}_{r_1}\left(\frac{ah_1}{k}\right)
\overline{B}_{r_2}\left(\frac{ah_2}{k}\right)
\end{equation*}
\begin{equation*}
=\frac{B_{r_1}B_{r_2}}{k^{r_1+r_2-1}} + \frac{(-1)^{(r_1-r_2)/2}
r_1r_2}{2^{r_1+r_2}k^{r_1+r_2-1}} \sum_{a=1}^{k-1}
\cot^{(r_1-1)}\left(\frac{\pi ah_1}{k}\right)
\cot^{(r_2-1)}\left(\frac{\pi ah_2}{k}\right).
\end{equation*}
\end{corollary}

\subsection{Generalized Hardy sums}

The Hardy sums (known also as Hardy--Berndt sums)
are defined for $h\in \Z$, $k\in \N$ as follows.
\begin{align*}
S(h,k) & := \sum_{a \text{ {\rm (mod $k$)}}} (-1)^{a+1+\lfloor ah/k\rfloor}, \\
s_1(h,k) & := \sum_{a \text{ {\rm (mod $k$)}}} (-1)^{\lfloor ah/k\rfloor} \left(\left(\frac{a}{k}\right)\right), \\
s_2(h,k) & := \sum_{a \text{ {\rm (mod $k$)}}} (-1)^a \left(\left(\frac{a}{k}\right)\right) \left(\left(\frac{ah}{k}\right)\right), \\
s_3(h,k) & := \sum_{a \text{ {\rm (mod $k$)}}} (-1)^a \left(\left(\frac{ah}{k}\right)\right), \\
s_4(h,k) & := \sum_{a \text{ {\rm (mod $k$)}}} (-1)^{\lfloor ah/k\rfloor}, \\
s_5(h,k) & :=\sum_{a \text{ {\rm (mod $k$)}}} (-1)^{a+\lfloor ah/k\rfloor} \left(\left(\frac{a}{k}\right)\right).
\end{align*}

B.~C.~Berndt and L.~A.~Goldberg \cite{BerGol1984} derived finite and infinite series representations
for the above sums. These identities were {{{{also}}}} obtained {{{{later}}}} by R.~Sitaramachandrarao \cite{Sit1987} by using some different
arguments. {{{{One could see}}}} \cite{Bec2003,BerGol1984,Die1984,Sit1987} for the history {{{{of these sums as well as}}}} for further results
on the Hardy sums, including reciprocity formulas.

We define the following generalization of $s_2(h,k)$:
\begin{equation*}
A(h_1,\ldots,h_m;k):= \sum_{\substack{a_1,\ldots,a_m \text{ {\rm (mod $k$)}} \\ a_1+\ldots+a_m\equiv 0 \text{ {\rm (mod $k$)}}}}
(-1)^{a_1} \left(\left(\frac{a_1h_1}{k}\right)\right)
\cdots \left(\left(\frac{a_m h_m}{k}\right)\right).
\end{equation*}

\begin{theorem} Let $k,m\in \N$ be even, $h_1,\ldots,h_m\in \Z$, $h_1$ odd, $\gcd(h_j,k)=1$ {\rm ($1\le j\le m$)}.
Then
\begin{equation*}
A(h_1,\ldots,h_m;k)= \frac{(-1)^{m/2-1}}{2^m k} \sum_{\substack{1\le a\le k-1\\ a\ne k/2}} \tan \left(\frac{\pi ah_1'}{k}\right)  \cot \left(\frac{\pi
ah_2'}{k}\right) \cdots \cot \left(\frac{\pi ah_m'}{k}\right).
\end{equation*}
\end{theorem}

\begin{proof} Let $f_1(n)=(-1)^n \left(\left(\frac{n}{k} \right)\right)$ and
$f_j(n)=\left(\left(\frac{n}{k} \right)\right)$ {\rm ($2\le j\le
m$)}. Apply Theorem \ref{th_main} and Lemma \ref{lemma}/(i),(iv).
\end{proof}

\begin{corollary} Assume that $m=2$. Let $k\in \N$ be even, $h_1,h_2\in \Z$, $h_1$ odd, and $\gcd(h_1,k)=\gcd(h_2,k)=1$.
Then
\begin{equation*}
\sum_{a=1}^{k-1} (-1)^a \left(\left(\frac{ah_1}{k} \right)\right) \left(\left(\frac{ah_2}{k} \right)\right)
= - \frac1{4k} \sum_{\substack{1\le a\le k-1\\ a\ne k/2}} \tan \left(\frac{\pi ah_2}{k}\right)
\cot \left(\frac{\pi ah_1}{k}\right).
\end{equation*}
\end{corollary}

In the case $m=2$, $h_1=1$, $h_2=h$ we obtain the following corollary, cf. \cite[Eq.\ (14)]{BerGol1984}, \cite[Eq.\ (7.3)]{Sit1987}.

\begin{corollary} If $k\in \N$ is even, $h\in \Z$, $\gcd(h,k)=1$, then
\begin{equation*}
s_2(h,k)= - \frac1{4k} \sum_{\substack{1\le a\le k-1\\ a\ne k/2}} \tan \left(\frac{\pi ah}{k}\right)
\cot \left(\frac{\pi a}{k}\right).
\end{equation*}
\end{corollary}

Next, we define the following common generalization of $s_1(h,k)$, $s_3(h,k)$ and $s_5(h,k)$, as follows.
\begin{equation*}
B(h_1,\ldots,h_m;k):= \sum_{\substack{a_1,\ldots,a_m \text{ {\rm (mod $k$)}} \\ a_1\not\equiv 0 \text{ {\rm (mod k)}}\\
a_1+\ldots+a_m\equiv 0 \text{ {\rm (mod $k$)}}}}
(-1)^{a_1h_1+k\lfloor a_1h_1/k \rfloor} \left(\left(\frac{a_2h_2}{k}\right)\right)
\cdots \left(\left(\frac{a_m h_m}{k}\right)\right).
\end{equation*}

\begin{theorem} Let $k\in \N$ be odd, $m\in \N$ be even, $h_j\in \Z$, $\gcd(h_j,k)=1$
{\rm ($1\le j\le m$)}. Then
\begin{equation*}
B(h_1,\ldots,h_m;k)= \frac{(-1)^{m/2}}{2^{m-1} k} \sum_{a=1}^{k-1} \tan \left(\frac{\pi ah_1'}{k}\right)  \cot \left(\frac{\pi
ah_2'}{k}\right) \cdots \cot \left(\frac{\pi ah_m'}{k}\right).
\end{equation*}
\end{theorem}

\begin{proof} Apply Theorem \ref{th_main} to the following functions:
\begin{align*}
f_1(n) & =  \begin{cases}  (-1)^{n \text{ {\rm (mod $k$)}}}, & \text{ if
} \  k\nmid n,
\\ 0, & \text{ if } \ k \mid n,
\end{cases} \\
f_j(n) & =\left(\left(\frac{n}{k} \right)\right) \quad (2\le j\le m)
\end{align*}
and also Lemma \ref{lemma}/(iv).
\end{proof}

\begin{corollary} Assume that $m=2$. Let $k\in \N$ be odd, $h_1,h_2\in \Z$, $h_1$ odd, $\gcd(h_1,k)=\gcd(h_2,k)=1$. Then
\begin{equation*}
\sum_{a=1}^{k-1} (-1)^{a+\lfloor ah_1/k\rfloor} \left(\left(\frac{ah_2}{k} \right)\right)
= \frac1{2k} \sum_{a=1}^{k-1} \tan \left(\frac{\pi ah_2}{k}\right)
\cot \left(\frac{\pi ah_1}{k}\right).
\end{equation*}
\end{corollary}

For $m=2$ in the special cases $h_1=1$, $h_2=h$ and $h_1=h$, $h_2=1$, respectively, we obtain the identities,
cf. \cite[Eq.\ (15), (17)]{BerGol1984}, \cite[Eq.\ (7.4), (7.6)]{Sit1987}, as follows.

\begin{corollary} If $k\in \N$ is odd, $h\in \Z$, $\gcd(h,k)=1$, then
\begin{equation} \label{repres_s_3}
s_3(h,k)= \frac1{2k} \sum_{a=1}^{k-1} \tan \left(\frac{\pi ah}{k}\right)
\cot \left(\frac{\pi a}{k}\right).
\end{equation}

If $k\in \N$ is odd, $h\in \Z$ is odd, $\gcd(h,k)=1$, then
\begin{equation*}
s_5(h,k)= \frac1{2k} \sum_{a=1}^{k-1} \tan \left(\frac{\pi a}{k}\right)
\cot \left(\frac{\pi ah}{k}\right).
\end{equation*}
\end{corollary}

\begin{corollary} Assume that $m=2$. Let $k\in \N$ be odd, $h_1,h_2\in \Z$, $h_1$ even, $\gcd(h_1,k)=\gcd(h_2,k)=1$. Then
\begin{equation*}
\sum_{a=1}^{k-1} (-1)^{\lfloor ah_1/k\rfloor} \left(\left(\frac{ah_2}{k} \right)\right)
= \frac1{2k} \sum_{a=1}^{k-1} \tan \left(\frac{\pi ah_2}{k}\right)
\cot \left(\frac{\pi ah_1}{k}\right).
\end{equation*}
\end{corollary}

For $m=2$ in the special case $h_1=h$ even, $h_2=1$  we obtain the identity,
cf. \cite[Eq.\ (13)]{BerGol1984}, \cite[Eq.\ (7.2)]{Sit1987}, as follows.

\begin{corollary} If $k\in \N$ is odd, $h\in \Z$ is even, $\gcd(h,k)=1$, then
\begin{equation} \label{repres_s_1}
s_1(h,k)= \frac1{2k} \sum_{a=1}^{k-1} \tan \left(\frac{\pi a}{k}\right)
\cot \left(\frac{\pi ah}{k}\right).
\end{equation}
\end{corollary}

Note that the Hardy sums $S(h,k)$ and $s_4(h,k)$ can also be treated with the DFT in the case {{{{when $k$ is}}}} odd. For example,
applying Corollary \ref{coroll_2} to the functions
\begin{equation*}
f_1(n)=f_2(n)=  \begin{cases}  (-1)^{n \text{ {\rm (mod $k$)}}}, & \text{ if
} \  k\nmid n,
\\ 0, & \text{ if } \ k \mid n
\end{cases}
\end{equation*}
we obtain the following representation.

\begin{corollary} If $k\in \N$ is odd, $h_1,h_2\in \Z$, $\gcd(h_1,k)=\gcd(h_2,k)=1$, then
\begin{equation} \label{form_s_4_gen}
\sum_{a=1}^{k-1} (-1)^{a(h_1+h_2) \text{ {\rm (mod $k$)}}} =  \frac1{k} \sum_{a=1}^{k-1} \tan \left(\frac{\pi ah_1}{k}\right)
\tan \left(\frac{\pi ah_2}{k}\right).
\end{equation}
\end{corollary}

If $h_2=1$ and $h_1=h$ is odd, then the left hand side of \eqref{form_s_4_gen} is exactly $s_4(h,k)$. See \cite[Eq.\ (16)]{BerGol1984},
\cite[Eq.\ (7.5)]{Sit1987}. If $h_1=h_2=1$, then \eqref{form_s_4_gen} provides the following classical identity, valid for $k\in \N$ odd,
cf. \cite[Prop.\ 3.1]{BecHal2010}:
\begin{equation*}
\sum_{a=1}^{k-1} \tan^2 \left(\frac{\pi a}{k} \right)= k^2-k.
\end{equation*}

\begin{remark} {\rm {{{{For $k$ odd, $h$ even}}}}, the formula \cite[Eq.\ (13)]{BerGol1984} receives the following representation:
\begin{equation*}
s_1(h,k)= - \frac1{2k} \sum_{\substack{j=1\\ j\ne (k+1)/2}}^k \cot \left(\frac{\pi h(2j-1)}{2k}\right)
\cot \left(\frac{\pi (2j-1)}{2k}\right),
\end{equation*}
which can easily be transformed into
\begin{equation*}
s_1(h,k)= \frac1{k} \sum_{j=1}^{(k-1)/2} \tan \left(\frac{\pi j}{k}\right)
\cot \left(\frac{\pi hj}{k}\right)
\end{equation*}
that is equal to the right hand side of \eqref{repres_s_1}. Similar considerations are valid for the corresponding
formulas on the Hardy sums $s_4(h,k)$, $s_5(h,k)$ and $S(h,k)$.}
\end{remark}

\begin{remark} {\rm The finite sum identities (7.1), (7.2), (7.3), (7.5), and (7.6) from the paper \cite{Sit1987} contain some misprints.
Namely, in formulas
(7.1) and (7.5) the sum $\sum_{r=1}^{k-1}$ should be $\sum_{r=1}^k$, while in (7.2) and (7.6) the {{{{sum}}}} $\sum_{r=1, r\ne (k+1)/2}^{k-1}$ should be
$\sum_{r=1, r\ne (k+1)/2}^k$, the missing terms being nonzero. Furthermore, in formula (7.3) the sum $\sum_{r=1}^{k-1}$ should be
$\sum_{r=1, r\ne k/2}^{k-1}$, the term for $r=k/2$ (namely $\tan (\pi/2)$) being not defined.}
\end{remark}

{{{{One could possibly}}}} investigate some further higher dimensional
generalizations and analogues of the Hardy sums involving the
Bernoulli functions, however we do not discuss this in the present paper.

\subsection{Sums involving the Hurwitz zeta function} \label{subsect_3_3}

Theorem \ref{th_main} and its corollaries can be applied in {{{{several}}}} other situations as well. For
example, let
\begin{equation*}
\zeta(s,x) :=\sum_{n=0}^{\infty} \frac1{(n+x)^s}
\end{equation*}
be the Hurwitz zeta function, where $0< x\le 1$ and
$\zeta(s,1)=\zeta(s)$ is the Riemann zeta function. The function
\begin{equation*}
D(h,k):=\sum_{a=1}^{k-1} \zeta \left(s_1,\left\{
\frac{ah_1}{k}\right\} \right) \zeta \left(s_2,\left
\{\frac{ah_2}{k} \right\}\right),
\end{equation*}
investigated by M.~Mikol\'as \cite{Mik1957}, is an analogue of the
Dedekind sum \eqref{Ded_Bern_2}, taking into account that
$$B_n(x)= -n\zeta(1-n,x) \quad (n\in \N,\ 0< x\le 1).$$

Let
\begin{equation} \label{period_zeta}
F(s,x) := \sum_{n=1}^{\infty} \frac{e^{\, 2\pi inx}}{n^s} \quad (x\in
{\Bbb R})
\end{equation}
be the periodic zeta function, which converges for $\Re s >0$ if $x\notin
\Z$ and for $\Re s>1$ if ${{{{x}}}}\in \Z$.

Applying Corollary \ref{cor_m_2_gen} to the functions $$f_j(n)=
F\left(s_j,\frac{n}{k}\right)\ \ (n\in \Z,\ 1\le j\le 2),$$
we deduce by Lemma \ref{lemma}/(v) the next new result.

\begin{theorem} Let $k\in \N$, $h_1,h_2 \in \Z$, $\gcd(h_1,k)=\gcd(h_2,k)=1$ and let
$s_1,s_2\in \C$, $\Re s_1, \Re s_2>1$. Then
\begin{equation*}
D(h,k)=(k^{s_1+s_2-1}-1)\zeta(s_1)\zeta(s_2)+k^{s_1+s_2-1}
\sum_{a=1}^{k-1} F\left(s_1,\frac{ah_2}{k}\right) F\left(s_2,-
\frac{ah_1}{k}\right).
\end{equation*}
\end{theorem}

\section{Some further remarks} \label{Sect_Final_remarks}

\setcounter{lemma}{1}
\setcounter{corollary}{11}

The following simple and useful result can be applied to obtain
infinite series representations for the Dedekind and Hardy sums.

\begin{lemma} \label{lemma_2} If $f:\N \to \C$ is a $k$-periodic {\rm ($k\in \N$)} odd function, then
\begin{align} \label{series_1}
S(f):= \sum_{r=1}^{\infty} \frac{f(r)}{r} & =  \frac{\pi}{2k}
\sum_{r=1}^{k-1} f(r) \cot \left(\frac{\pi r}{k}\right) \\
\label{series_2} & = - \frac{\pi i}{k^2} \sum_{r=1}^{k-1} r
\widehat{f}(r).
\end{align}
\end{lemma}

For the Dedekind sum, \eqref{inf_repres} is a direct consequence of identities
\eqref{series_1} and \eqref{cot_repres}. As another example, \eqref{series_1} and  \eqref{repres_s_3} imply that for
$k$ odd, $\gcd(h,k)=1$, {{{{one has}}}}
\begin{equation*}
s_3(h,k)  = \frac1{\pi} \sum_{r=1}^{\infty} \frac1{r} \tan
\left(\frac{\pi rh}{k} \right).
\end{equation*}

One could see \cite[Th.\ 1]{BerGol1984}, \cite[Th. 7.1]{Sit1987} for the above formula as well as for similar representations regarding Hardy sums.

Identity \eqref{series_1} of Lemma \ref{lemma_2} was proved in
\cite[Lemma 2.1]{Sit1987} by applying results of D.~H.~Lehmer
\cite{Leh1975} on the generalized Euler constants $\gamma(r,k)$
associated with the infinite arithmetic progression $r,r+k,r+2k,\ldots$
($1\le r\le k$), where
\begin{equation*}
\gamma(r,k):=\lim_{x\to \infty} \left( \sum_{\substack{1\le n\le x\\
n\equiv r \ {\text {\rm (mod $k$)}}}} \frac1{n} -\frac1{k}\log x
\right).
\end{equation*}

B.~C.~Berndt \cite{Ber1979} deduced \eqref{series_2} by contour
integration (with a different definition of the DFT). The fact that
the finite sums \eqref{series_1} and \eqref{series_2} are equal, {{{{provides}}}}
another simple consequence of Corollary \ref{coroll_2}, applied
to the odd functions $f:\N \to \C$ and $n\mapsto \left(\left(
\frac{n}{k}\right)\right)$.

Furthermore, according to \cite[Th.\ 8]{Leh1975}, if $f$ is a
$k$-periodic function, then
\begin{equation*}
S(f) = \sum_{r=1}^k f(r)\gamma(r,k),
\end{equation*}
provided that $\sum_{r=1}^k f(k)=0$, representing a necessary and
sufficient condition for the convergence of the series $S(f)$ (which holds if
$f$ is a $k$-periodic and odd function).

We note that the DFT of the $k$-periodic function $r\mapsto
\gamma(r,k)$ is
\begin{equation*}
\widehat{\gamma}(r,k) = \begin{cases} F\left(1,-\frac{r}{k}\right),
& \text{ if } \ k\nmid r, \\ \gamma, & \text{ if } \ k\mid r
\end{cases}
\end{equation*}
(cf. \cite[p.\ 127]{Leh1975}), where $F(s,x)$ {{{{is}}}} the periodic zeta function defined by
\eqref{period_zeta} and $\gamma:=\gamma(0,1)$ is Euler's constant.
Therefore, we deduce by Corollary \ref{cor_m_2_gen} the next identity.

\begin{corollary} If $f:\N \to \C$ is a $k$-periodic {\rm ($k\in \N$)} odd function, then
\begin{equation*}
S(f)= - \frac1{k} \sum_{r=1}^{k-1} \widehat{f}(r)
F\left(1,-\frac{r}{k}\right).
\end{equation*}
\end{corollary}

{\bf Acknowledgement.} The authors would like to thank the referee for useful remarks which helped improve the presentation of the paper.


\vskip3mm

\noindent Michael Th. Rassias \\ Department of Mathematics, ETH-Z\"{u}rich, R\"{a}mistrasse 101,
8092 Z\"{u}rich, Switzerland \\ Department of Mathematics, Princeton University, Fine Hall,
Washington Road, Princeton, NJ 08544-1000, USA \\
E-mail: {\tt michail.rassias@math.ethz.ch,
michailrassias@math.princeton.edu}

\vskip3mm

\noindent L\'aszl\'o T\'oth \\ Department of Mathematics, University of P\'ecs, Ifj\'us\'ag \'utja
6, 7624 P\'ecs, Hungary \\
E-mail: {\tt ltoth@gamma.ttk.pte.hu}

\end{document}